\documentclass[11pt]{article}

\usepackage{amssymb,amsmath,amsfonts,amsthm}
\usepackage{latexsym}
\usepackage{graphics}
\usepackage{indentfirst}

\newtheorem{innercustomthm}{Theorem}
\newenvironment{customthm}[1]
  {\renewcommand\theinnercustomthm{#1}\innercustomthm}
  {\endinnercustomthm}

\newtheorem{innercustomcor}{Corollary}
\newenvironment{customcor}[1]
  {\renewcommand\theinnercustomcor{#1}\innercustomcor}
  {\endinnercustomcor}
\newtheorem{innercustompro}{Proposition}
\newenvironment{custompro}[1]
  {\renewcommand\theinnercustompro{#1}\innercustompro}
  {\endinnercustompro}
\newtheorem*{thm*}{Theorem}
\newtheorem{thm}{Theorem}
\newtheorem{lem}[thm]{Lemma}
\newtheorem{pro}[thm]{Proposition}

\newtheorem{cor}[thm]{Corollary}

\newtheorem{ques}[thm]{Question}

\newcommand{\N}{\mathbb{N}}

\newcommand{\col}{\mathrm{col}}

\begin{document}

\title{On the Chromatic Polynomial and Counting DP-Colorings of Graphs}

\author{Hemanshu Kaul\footnote{Department of Applied Mathematics, Illinois Institute of Technology, Chicago, IL 60616. E-mail: {\tt kaul@iit.edu}} \\
Jeffrey A. Mudrock\footnote{Department of Mathematics, College of Lake County, Grayslake, IL 60030. E-mail: {\tt jmudrock@clcillinois.edu}} }

\maketitle

\begin{abstract}

The chromatic polynomial of a graph $G$, denoted $P(G,m)$, is equal to the number of proper $m$-colorings of $G$.  The list color function of graph $G$, denoted $P_{\ell}(G,m)$, is a list analogue of the chromatic polynomial that has been studied since 1992, primarily through comparisons with the corresponding chromatic polynomial. DP-coloring (also called correspondence coloring) is a generalization of list coloring recently introduced by  Dvo\v{r}\'{a}k and Postle.  In this paper, we introduce a DP-coloring analogue of the chromatic polynomial called the DP color function, denoted $P_{DP}(G,m)$, and ask several fundamental open questions about it, making progress on some of them. Motivated by known results related to the list color function, we show that while the DP color function behaves similar to the list color function for some graphs, there are also some surprising differences. For example, Wang, Qian, and Yan recently showed that if $G$ is a connected graph with $l$ edges, then $P_{\ell}(G,m)=P(G,m)$ whenever $m > \frac{l-1}{\ln(1+ \sqrt{2})}$, but we will show that for any $g \geq 3$ there exists a graph, $G$, with girth $g$ such that $P_{DP}(G,m) < P(G,m)$ when $m$ is sufficiently large. We also study the asymptotics of $P(G,m) - P_{DP}(G,m)$ for a fixed graph $G$. We develop techniques to compute $P_{DP}(G,m)$ exactly and apply them to certain classes of graphs such as chordal graphs, unicyclic graphs, and cycles with a chord.  Finally, we make progress towards showing that for any graph $G$, there is a $p$ such that $P_{DP}(G \vee K_p, m) = P(G \vee K_p , m)$ for large enough $m$.

\medskip



\end{abstract}

\section{Introduction}\label{intro}

In this paper all graphs are nonempty, finite, simple graphs unless otherwise noted.  Generally speaking we follow West~\cite{W01} for terminology and notation.  The set of natural numbers is $\N = \{1,2,3, \ldots \}$.  Given a set $A$, $\mathcal{P}(A)$ is the power set of $A$.  For $m \in \N$, we write $[m]$ for the set $\{1,2, \ldots, m \}$.  If $G$ is a graph and $S, U \subseteq V(G)$, we use $G[S]$ for the subgraph of $G$ induced by $S$, and we use $E_G(S, U)$ for the subset of $E(G)$ with one endpoint in $S$ and one endpoint in $U$.  For $v \in V(G)$, we write $d_G(v)$ for the degree of vertex $v$ in the graph $G$, and we write $N_G(v)$ (resp. $N_G[v]$) for the neighborhood (resp. closed neighborhood) of vertex $v$ in the graph $G$.  If $G$ and $H$ are vertex disjoint graphs, we write $G \vee H$ for the join of $G$ and $H$.

\subsection{The Chromatic Polynomial and the List Color Function}

List coloring is a well known variation on the classic vertex coloring problem, and it was introduced independently by Vizing~\cite{V76} and Erd\H{o}s, Rubin, and Taylor~\cite{ET79} in the 1970's.  In the classic vertex coloring problem we wish to color the vertices of a graph $G$ with up to $m$ colors from $[m]$ so that adjacent vertices receive different colors, a so-called \emph{proper $m$-coloring}. The chromatic number of a graph, denoted $\chi(G)$, is the smallest $m$ such that $G$ has a proper $m$-coloring.  For list coloring, we associate a \emph{list assignment}, $L$, with a graph $G$ such that each vertex $v \in V(G)$ is assigned a list of colors $L(v)$ (we say $L$ is a list assignment for $G$).  The graph $G$ is \emph{$L$-colorable} if there exists a proper coloring $f$ of $G$ such that $f(v) \in L(v)$ for each $v \in V(G)$ (we refer to $f$ as a \emph{proper $L$-coloring} of $G$).  A list assignment $L$ is called an \emph{m-assignment} for $G$ if $|L(v)|=m$ for each $v \in V(G)$.  The \emph{list chromatic number} of a graph $G$, denoted $\chi_\ell(G)$, is the smallest $m$ such that $G$ is $L$-colorable whenever $L$ is an $m$-assignment for $G$.  It is immediately obvious that for any graph $G$, $\chi(G) \leq \chi_\ell(G)$.  We say $G$ is \emph{$m$-choosable} if $m \geq \chi_\ell(G)$.

For $m \in \N$, we let $P(G,m)$ be the \emph{chromatic polynomial} of the graph $G$; that is, $P(G,m)$ is equal to the number of proper $m$-colorings of $G$.  It can be shown that $P(G,m)$ is a polynomial in $m$ of degree $|V(G)|$ (see~\cite{B12}).  For example, $P(K_n,m) = \prod_{i=0}^{n-1} (m-i)$, $P(C_n,m) = (m-1)^n + (-1)^n (m-1)$, and $P(T,m) = m(m-1)^{n-1}$ whenever $T$ is a tree on $n$ vertices (see~\cite{W01}).

The notion of chromatic polynomial was extended to list coloring as follows. If $L$ is a list assignment for $G$, we use $P(G,L)$ to denote the number of proper $L$-colorings of $G$. The \emph{list color function} $P_\ell(G,m)$ is the minimum value of $P(G,L)$ where the minimum is taken over all possible $m$-assignments $L$ for $G$.  Since an $m$-assignment could assign the same $m$ colors to every vertex in a graph, it is clear that $P_\ell(G,m) \leq P(G,m)$ for each $m \in \N$.  In general, the list color function can differ significantly from the chromatic polynomial for small values of $m$.  One reason for this is that a graph can have a list chromatic number that is much higher than its chromatic number.  Indeed, for each $m \in \N$, $\chi_\ell(K_{m,t})=m+1$ if and only if $t \geq m^m$.   On the other hand, for large values of $m$, Wang, Qian, and Yan~\cite{WQ17} (improving upon results in \cite{D92} and \cite{T09}) recently showed the following.

\begin{thm} [\cite{WQ17}] \label{thm: WQ17}
If $G$ is a connected graph with $l$ edges, then $P_{\ell}(G,m)=P(G,m)$ whenever $m > \frac{l-1}{\ln(1+ \sqrt{2})}$.
\end{thm}

It is also known that $P_{\ell}(G,m)=P(G,m)$ for all $m \in \N$ when $G$ is a cycle or chordal
(see~\cite{AS90} and~\cite{KN16}).  Moreover, if $P_{\ell}(G,m)=P(G,m)$ for all $m \in \N$, then $P_{\ell}(G \vee K_n,m)=P(G \vee K_n,m)$ for each $n, m \in \N$ (see~\cite{KM18}). Thomassen~\cite{T09} gives a survey of known results and open questions on the list color function.

\subsection{DP-Coloring}

In 2015, Dvo\v{r}\'{a}k and Postle~\cite{DP15} introduced DP-coloring (they called it correspondence coloring) in order to prove that every planar graph without cycles of lengths 4 to 8 is 3-choosable.  DP-coloring has been extensively studied over the past 4 years (see e.g.,~\cite{B16,B17,BK182,BK19,BK17,BK18, BK172, KO18, KO182, KY17, LL19, LLYY19, Mo18, M18}). Intuitively, DP-coloring is a generalization of list coloring where each vertex in the graph still gets a list of colors but identification of which colors are different can vary from edge to edge.  Following~\cite{BK17}, we now give the formal definition.  Suppose $G$ is a graph.  A \emph{cover} of $G$ is a pair $\mathcal{H} = (L,H)$ consisting of a graph $H$ and a function $L: V(G) \rightarrow \mathcal{P}(V(H))$ satisfying the following four requirements:

\vspace{5mm}

\noindent(1) the set $\{L(u) : u \in V(G) \}$ is a partition of $V(H)$; \\
(2) for every $u \in V(G)$, the graph $H[L(u)]$ is complete; \\
(3) if $E_H(L(u),L(v))$ is nonempty, then $u=v$ or $uv \in E(G)$; \\
(4) if $uv \in E(G)$, then $E_H(L(u),L(v))$ is a matching (the matching may be empty).

\vspace{5mm}

Suppose $\mathcal{H} = (L,H)$ is a cover of $G$.  We say $\mathcal{H}$ is \emph{$m$-fold} if $|L(u)|=m$ for each $u \in V(G)$.  An $\mathcal{H}$-coloring of $G$ is an independent set in $H$ of size $|V(G)|$.  It is immediately clear that an independent set $I \subseteq V(H)$ is an $\mathcal{H}$-coloring of $G$ if and only if $|I \cap L(u)|=1$ for each $u \in V(G)$.  The \emph{DP-chromatic number} of a graph $G$, $\chi_{DP}(G)$, is the smallest $m \in \N$ such that $G$ admits an $\mathcal{H}$-coloring for every $m$-fold cover $\mathcal{H}$ of $G$.

Given an $m$-assignment, $L$, for a graph $G$, it is easy to construct an $m$-fold cover $\mathcal{H}$ of $G$ such that $G$ has an $\mathcal{H}$-coloring if and only if $G$ has a proper $L$-coloring (see~\cite{BK17}).  It follows that $\chi_\ell(G) \leq \chi_{DP}(G)$.  This inequality may be strict since it is easy to prove that $\chi_{DP}(C_n) = 3$ whenever $n \geq 3$, but the list chromatic number of any even cycle is 2 (see~\cite{BK17} and~\cite{ET79}).

We now briefly discuss some similarities between DP-coloring and list coloring.  First, notice that like $m$-choosability, the graph property of having DP-chromatic number at most $m$ is monotone. The \emph{coloring number} of a graph $G$, denoted $\col(G)$, is the smallest integer $d$ for which there exists an ordering, $v_1, v_2, \ldots, v_n$, of the elements in $V(G)$ such that each vertex $v_i$ has at most $d-1$ neighbors among $v_1, v_2, \ldots, v_{i-1}$.  It is easy to prove that $\chi_\ell(G) \leq \chi_{DP}(G) \leq \col(G)$. Molloy~\cite{Mo18} has shown that Kahn's~\cite{K96} seminal result that list edge-chromatic number of a simple graph asymptotically equals to the edge-chromatic number holds for DP-coloring as well. Thomassen~\cite{T94} famously proved that every planar graph is 5-choosable, and Dvo\v{r}\'{a}k and Postle~\cite{DP15} observed that the DP-chromatic number of every planar graph is at most 5.  Also, Molloy~\cite{M17} recently improved a theorem of Johansson by showing that every triangle-free graph $G$ with maximum degree $\Delta(G)$ satisfies $\chi_\ell(G) \leq (1 + o(1)) \Delta(G)/ \log(\Delta(G))$.  Bernshteyn~\cite{B17} subsequently showed that this bound also holds for the DP-chromatic number.  DP-coloring has also been used to prove results about list coloring.  Indeed, the original motivation for DP-coloring was a list coloring problem~\cite{DP15}, and in general, DP-coloring can provide an advantage in inductive proofs (over working directly with list coloring) in that it provides a stronger inductive hypothesis which allows for more flexibility in the proof (see e.g.~\cite{BK182}).

On the other hand, Bernshteyn~\cite{B16} showed that if the average degree of a graph $G$ is $d$, then $\chi_{DP}(G) = \Omega(d/ \log(d))$.  This is in stark contrast to the celebrated result of Alon~\cite{A00} which says $\chi_\ell(G) = \Omega(\log(d))$.  It was also recently shown in~\cite{BK17} that there exist planar bipartite graphs with DP-chromatic number 4 even though the list chromatic number of any planar bipartite graph is at most 3~\cite{AT92}.  A famous result of Galvin~\cite{G95} says that if $G$ is a bipartite multigraph and $L(G)$ is the line graph of $G$, then $\chi_\ell(L(G)) = \chi(L(G)) = \Delta(G)$.  However, it is also shown in~\cite{BK17} that every $d$-regular graph $G$ satisfies $\chi_{DP}(L(G)) \geq d+1$.

\subsection{The DP Color Function}

The following definition is now quite natural.  Suppose $\mathcal{H} = (L,H)$ is a cover of graph $G$.  We let $P_{DP}(G, \mathcal{H})$ be the number of $\mathcal{H}$-colorings of $G$.  Then, the \emph{DP color function}, denoted $P_{DP}(G,m)$, is the minimum value of $P_{DP}(G, \mathcal{H})$ where the minimum is taken over all possible $m$-fold covers $\mathcal{H}$ of $G$.  Based upon what is discussed above, we immediately have that for any graph $G$ and $m \in \N$,
$$ P_{DP}(G, m) \leq P_\ell(G,m) \leq P(G,m).$$
Note that if $G$ is a disconnected graph with components: $H_1, H_2, \ldots, H_t$, then $P_{DP}(G, m) = \prod_{i=1}^t P_{DP}(H_i,m)$ \footnote{An analogous property holds for the chromatic polynomial and list color function.}.  So, understanding the DP color function of $G$ amounts to understanding the DP color function of its components.  Due to this fact, we will only consider connected graphs from this point forward unless otherwise noted.

We now present two key questions that led to the results in this paper.  Based upon known results for the list color function, the following question is natural.

\begin{ques} \label{ques: 1}
For every graph $G$, does $P_{DP}(G, m) = P(G,m)$ for sufficiently large $m$?
\end{ques}

Perhaps surprisingly, the answer to Question~\ref{ques: 1} is no in a fairly strong sense.  We will see below that $P_{DP}(G, m) < P(G,m)$ for sufficiently large $m$ whenever $G$ is a graph with girth that is even.  Another natural question, that will be partially addressed in this paper, is:

\begin{ques} \label{ques: 2}
For which graphs $G$ does $P_{DP}(G,m) = P(G,m)$ for every $m \in \N$?
\end{ques}

One could also ask for comparison of $P_{DP}(G,m)$, $P_{\ell}(G, m)$, and $P(G,m)$ for small values of $m$.  Additionally, it is possible for the DP color function to be a useful tool for pursuing open questions about the list color function since it bounds the list color function from below.  For example, Thomassen~\cite{T09} asked if there exists a graph $G$ and an $m > 2$ such that $P_{\ell}(G,m) = 1$.  Clearly, one could make progress on this question by showing $P_{DP}(G,m) > 1$ for certain $G$ and $m \in \N$.

\subsection{Outline of Results and further Open Questions}

We now present an outline of the paper while also mentioning some open questions.

We begin Section~\ref{asymptotics} by proving for any $G$ and $m \in \N$: $P_{DP}(G,m) \le  \frac{m^{|V(G)|} (m-1)^{|E(G)|}}{m^{|E(G)|}}$ which is the same as the lower bound on $P(G,m)$ when $G$ is bipartite, as claimed by the well-known Sidorenko's conjecture on counting homomorphisms from a bipartite graph \footnote{When $|E(G)|=0$, we take $0^0$ to equal 1} (see~\cite{CL17} for a proof of this restriction of Sidorenko's conjecture and citations therein). It would be natural to ask whether the same lower bound also holds for DP color function of bipartite graphs, but our upper bound shows that such a conjecture would be possible only if $P_{DP}(G,m) = \frac{m^{|V(G)|} (m-1)^{|E(G)|}}{m^{|E(G)|}}$ for bipartite $G$.  We will use this upper bound along with Theorems~\ref{thm: chordal} and~\ref{thm: onecycle} to prove that this form of Sidorenko's conjecture for the DP color function holds only for trees.

\begin{cor} \label{cor: sid}  For any connected graph $G$, $P_{DP}(G,m) = \frac{m^{|V(G)|} (m-1)^{|E(G)|}}{m^{|E(G)|}}$ for all $m \in \N$ if and only if $G$ is a tree.
\end{cor}

We also use this upper bound to prove the following result.

\begin{thm} \label{thm: evengirth}
If $G$ is a graph with girth $g$ such that $g$ is even, then there is an $N \in \N$ such that $P_{DP}(G,m) < P(G,m)$ whenever $m \geq N$.
\end{thm}

In contrast, we also show that a unicyclic graph with an odd cycle has $P_{DP}(G,m) = P(G,m)$ (see Theorem~\ref{thm: onecycle} below). With these results in mind, one might conjecture that all graphs with girth that is odd have a DP color function that eventually equals its chromatic polynomial.  However, we will also show the following in Section~\ref{asymptotics}.

\begin{cor} \label{cor: girth}
For any integer $g \geq 3$ there exists a graph $G$ with girth $g$ and an $N \in \N$ such that $P_{DP}(G,m) < P(G,m)$ whenever $m \geq N$.
\end{cor}

Since there are examples of large families of graphs that have a DP color function that is eventually strictly smaller than the corresponding chromatic polynomial, the following question is natural.

\begin{ques} \label{ques: polynomial}
For any graph $G$ does there always exist an $N \in \N$ and a polynomial $p(m)$ such that $P_{DP}(G,m) = p(m)$ whenever $m \geq N$?
\end{ques}

It is also natural to study the asymptotics of $P(G,m) - P_{DP}(G,m)$ for any graph $G$.  We end Section~\ref{asymptotics} by studying $P(G,m) - P_{DP}(G,m)$ for arbitrary $G$.  In particular, we prove the following.

\begin{thm} \label{thm: asymptotic}
For any graph $G$ with $n$ vertices,
$$P(G,m)-P_{DP}(G,m) = O(m^{n-2}) \; \; \text{as $m \rightarrow \infty$.}$$
\end{thm}

Interestingly, we do not have an example of a graph $G$ such that $P(G,m)-P_{DP}(G,m) = \Theta(m^{n-2})$.  If $G$ is a unicyclic graph on $n$ vertices that contains a cycle of length $4$, we will see below that $P(G,m)-P_{DP}(G,m) = \Theta(m^{n-3})$ (see Theorem~\ref{thm: onecycle}).  This leads us to pose the following question.

\begin{ques} \label{ques: asymptotic}
For any graph $G$ with $n$ vertices, is it the case that $P(G,m)-P_{DP}(G,m) = O(m^{n-3})$ as $m \rightarrow \infty$?
\end{ques}

In Section~\ref{chordal} we show that in contrast to previous results, the DP color function of any chordal graph behaves just like the list color function of the graph (see~\cite{AS90}).

\begin{thm} \label{thm: chordal}
If $G$ is chordal, then $P_{DP}(G,m) = P(G,m)$ for every $m \in \N$.
\end{thm}

Notice that Theorem~\ref{thm: chordal} tells us that graphs of infinite girth (i.e. forests) have a DP color function that equals the corresponding chromatic polynomial for all natural numbers. In Section~\ref{bijections}, this observation motivates a notion of natural bijections between DP-colorings and proper colorings of a graph, and we use this notion to develop some tools that are useful for exactly determining the DP color function. These tools prove particularly useful for studying graphs that are a few edges away from being trees. In Section~\ref{exact} we use these tools to find formulas for the DP color function of connected graphs containing one cycle (i.e., unicyclic graphs).

\begin{thm} \label{thm: onecycle}
Suppose $G$ is a unicyclic graph on $n$ vertices.  Then the following statements hold.
\\
(i) For $m \in \N$, if $G$ contains a cycle on $2k+1$ vertices, then $P_{DP}(G,m) = P(G,m)$.
\\
(ii)  For $m \geq 2$, if $G$ contains a cycle on $2k+2$ vertices, then $$P_{DP}(G,m) = (m-1)^n - (m-1)^{n-2k-2}.$$
\end{thm}

In Section~\ref{exact} we also find exact formulas for the DP color function of a cycle plus a chord, with the answer depending on the parity of the lengths of the two maximal cycles properly contained in such a graph.

Finally, in Section~\ref{join}, we study the DP color function of the join of a complete graph and arbitrary graph $G$. Recently, Bernshteyn, Kostochka, and Zhu~\cite{BK18} showed that for any graph $G$ there exists an $N \leq 3|E(G)|$ such that $\chi_{DP}(G \vee K_p) = \chi(G \vee K_p)$ whenever $p \geq N$. So, it is natural to ask whether taking the join of an arbitrary graph with an appropriate clique makes the chromatic polynomial equal to the DP color function.

\begin{ques} \label{ques: join}
For every graph $G$, does there exist $p,N \in \N$ such that $P_{DP}(G \vee K_p, m) = P(G \vee K_p , m)$ whenever $m \geq N$?
\end{ques}

In Section~\ref{join}, we prove both of the following results the second of which is a partial answer to Question~\ref{ques: join}.

\begin{thm} \label{thm: joinbound}
Suppose $G$ is a graph with $\col(G) \geq 3$.  Then for $m \geq \col(G) + 3$,
$$P_{DP}(K_1 \vee G, m) \geq \min \{P(K_1 \vee G,m), mP_{DP}(G, m-1) + 2(m-\col(G)-2)^{|V(G)|-2} \}.$$
\end{thm}

This is easily generalizable to a lower bound for $P_{DP}(K_p \vee G, m)$, and leads to the following.

\begin{thm} \label{thm: joinasy}
Suppose $G$ is a graph on $n$ vertices such that
$$P(G,m)-P_{DP}(G,m) = O(m^{n-3}) \; \; \text{as $m \rightarrow \infty$.}$$
Then, there exist $p,N \in \N$ such that $P_{DP}(G \vee K_p, m) = P(G \vee K_p , m)$ whenever $m \geq N$.
\end{thm}

Notice that if the answer to Question~\ref{ques: asymptotic} is yes, then Theorem~\ref{thm: joinasy} would imply that the answer to Question~\ref{ques: join} is yes.

We end this section with one final open question as a follow-up to Questions~\ref{ques: 1} and~\ref{ques: 2} from the previous section.

\begin{ques} \label{ques: mono}
For a graph $G$ such that $P_{DP}(G,m_0) = P(G,m_0)$ for some $m_0 \geq \chi(G)$, is $P_{DP}(G,m) = P(G,m)$ for all $m \geq m_0$?
\end{ques}

The corresponding question for the list color function is also open (see Question~2 in~\cite{KN16}).

\section{Girth Parity and Asymptotics} \label{asymptotics}

We begin this section by using a simple probabilistic argument to obtain an upper bound on the DP color function of an arbitrary graph.

\begin{pro} \label{pro: DPgenupper}
Suppose $G$ is a graph with vertex set $V(G) = \{v_1, \ldots, v_n \}$.  Then, for each $m \geq 1$,
$$P_{DP}(G,m) \leq \frac{m^{n} (m-1)^{|E(G)|}}{m^{|E(G)|}}.$$
\end{pro}

\begin{proof}
The result is obvious when $G$ is edgeless.  So, we suppose throughout the proof that $|E(G)| \geq 1$.  We form an $m$-fold cover, $\mathcal{H} = (L,H)$, of $G$ by the following (partially random) process. We begin by letting $L(v_i) = \{(v_i, j) : j \in [m] \}$ for each $i \in [n]$.  Let the graph $H$ have vertex set $\bigcup_{i=1}^{n} L(v_i)$.  Also, draw edges in $H$ so that $H[L(v)]$ is a clique for each $v \in V(G)$.  Finally, for each $uv \in E(G)$, uniformly at random choose a perfect matching between $L(u)$ and $L(v)$ from the $m!$ possible perfect matchings.  It is easy to see that $\mathcal{H}=(L,H)$ is an $m$-fold cover of $G$.

Let $\mathcal{I} = \{ I \subset V(H) : |I \cap L(v)|=1 \; \; \text{for each} \; \; v \in V(G) \}$.  Clearly, all $\mathcal{H}$-colorings of $G$ are contained in $\mathcal{I}$ and $|\mathcal{I}| = m^{n}$.  Suppose we index the elements of $\mathcal{I}$ so that $\mathcal{I} = \{I_1, I_2, \ldots, I_{m^{n}} \}$.  For each $j \in [m^{n}]$, let $E_j$ be the event that $I_j$ is an $\mathcal{H}$-coloring of $G$.  Notice that if $uv \in E(G)$, then the probability that $a \in L(u)$ and $b \in L(v)$ are not adjacent in $H$ is $(1 - 1/m)$.  In order for event $E_j$ to occur, for each $uv \in E(G)$, the vertex in $I_j \cap L(u)$ must not be adjacent (in $H$) to the vertex in $I_j \cap L(v)$.  So,
$$P[E_j] = \left(1- \frac{1}{m} \right)^{|E(G)|}.$$
Now, let $X_j$ be the random variable that is one if $E_j$ occurs and zero otherwise.  Let $X = \sum_{j=1}^{m^{n}} X_j$.  Notice that $X$ is the random variable equal to the number of $\mathcal{H}$-colorings of $G$.  By linearity of expectation, we have that
$$E[X] =  \sum_{j=1}^{m^{n}} E[X_j] = m^{n} \left(1- \frac{1}{m} \right)^{|E(G)|} = \frac{m^{n} (m-1)^{|E(G)|}}{m^{|E(G)|}}.$$
The result follows.
\end{proof}

Note that this upper bound is the same as the lower bound on $P(G,m)$ when $G$ is bipartite, as claimed by the well-known Sidorenko's conjecture on counting homomorphisms from bipartite graphs (see~\cite{CL17} for a proof of this restriction of Sidorenko's conjecture and citations therein). So, Proposition~\ref{pro: DPgenupper} shows that Sidorenko's conjecture for the DP color function of bipartite graphs would be possible only if $P_{DP}(G,m) = \frac{m^{n} (m-1)^{|E(G)|}}{m^{|E(G)|}}$ for bipartite $G$. Proposition~\ref{pro: DPgenupper} along with Theorems~\ref{thm: chordal} and~\ref{thm: onecycle} (from Sections~\ref{chordal} and ~\ref{exact}) gives us Corollary~\ref{cor: sid} that shows Sidorenko's conjecture for DP color function holds only for trees.

\begin{customcor} {\bf \ref{cor: sid}}
For any connected graph $G$, $P_{DP}(G,m) = \frac{m^{|V(G)|} (m-1)^{|E(G)|}}{m^{|E(G)|}}$ for all $m \in \N$ if and only if $G$ is a tree.
\end{customcor}

\begin{proof}
The ``if" direction is implied by Theorem~\ref{thm: chordal}.  Conversely, let $s= \frac{m^{|V(G)|} (m-1)^{|E(G)|}}{m^{|E(G)|}}$, and suppose that $G$ is a connected graph chosen so that $P_{DP}(G,m)=s$.  If $|E(G)| = |V(G)|$, Theorem~\ref{thm: onecycle} tells us that  $P_{DP}(G,m) < s$.  Similarly, if $|E(G)| \geq |V(G)|+1$, Proposition~\ref{pro: DPgenupper} and the fact that for $m \geq 2$, $s$ is not an integer implies $P_{DP}(G,m) < s$ whenever $m \geq 2$.  So, if $P_{DP}(G,m) = s$, then $|E(G)| = |V(G)|-1$ which implies $G$ is a tree.
\end{proof}

The next result, which follows easily from Whitney's Broken Circuit Theorem~\cite{W32}, along with Proposition~\ref{pro: DPgenupper} are the key results that will be used in the proof of Theorem~\ref{thm: evengirth}.

\begin{pro} \label{pro: coefficients}
Suppose $G$ is a connected graph on $n$ vertices and $s$ edges having girth $g \in \N$.  Suppose $P(G,m) = \sum_{i=0}^{n} (-1)^i a_i m^{n-i}$.  Then, for $i = 0, 1, \ldots, g-2$
$$ a_i = \binom{s}{i} \; \; \text{and} \; \; a_{g-1} = \binom{s}{g-1} - t$$
where $t$ is the number of cycles of length $g$ contained in $G$.
\end{pro}

We are now ready to prove Theorem~\ref{thm: evengirth}

\begin{customthm} {\bf \ref{thm: evengirth}}
If $G$ is a graph with girth $g$ that is even, then there is an $N \in \N$ such that $P_{DP}(G,m) < P(G,m)$ whenever $m \geq N$.
\end{customthm}

\begin{proof}
WLOG let $G$ be connected. We know that $G$ has at least $g$ edges and at least $g$ vertices.  Let $n=|V(G)|$.  By Proposition~\ref{pro: DPgenupper}, we know that for $m \geq 1$,
$$P(G,m) - \frac{m^{n} (m-1)^{|E(G)|}}{m^{|E(G)|}} \leq P(G,m) - P_{DP}(G,m).$$
Suppose that $P(G,m) = \sum_{i=0}^{n} (-1)^i a_i m^{n-i}$, and $t$ is the number of cycles of length $g$ contained in $G$.  Clearly, $t \geq 1$.  Applying the binomial theorem and Proposition~\ref{pro: coefficients}, we obtain:
\begin{align*}
&P(G,m) - \frac{m^{n} (m-1)^{|E(G)|}}{m^{|E(G)|}} \\
&= \sum_{i=0}^{n} (-1)^i a_i m^{n-i} - m^{n-|E(G)|} \sum_{i=0}^{|E(G)|} (-1)^i \binom{|E(G)|}{i} m^{|E(G)|-i} \\
&= \sum_{i=g}^{n} (-1)^i a_i m^{n-i} + \sum_{i=0}^{g-1} (-1)^i a_i m^{n-i} - \sum_{i=0}^{|E(G)|} (-1)^i \binom{|E(G)|}{i} m^{n-i} \\
&= \sum_{i=g}^{n} (-1)^i a_i m^{n-i} + \sum_{i=0}^{g-1} (-1)^i \left(a_i - \binom{|E(G)|}{i} \right) m^{n-i} - \sum_{i=g}^{|E(G)|} (-1)^i \binom{|E(G)|}{i} m^{n-i} \\
&= (-1)^{g-1} (-t) m^{n-g+1} +  \sum_{i=g}^{n} (-1)^i a_i m^{n-i} - \sum_{i=g}^{|E(G)|} (-1)^i \binom{|E(G)|}{i} m^{n-i}.
\end{align*}
Since we know that $g$ is even, $t m^{n-g+1}$ is the dominant term of $P(G,m) - m^{n-|E(G)|} (m-1)^{|E(G)|}$.  So, there is a natural number $N$ such that $0 < P(G,m) - m^{n-|E(G)|} (m-1)^{|E(G)|}$ whenever $m \geq N$.  The result follows.
\end{proof}

We will now present a result that will allow us to construct a graph, $G$, with girth equal to any odd number, satisfying $P_{DP}(G,m) < P(G,m)$ for sufficiently large $m$.  We begin with a definition.  Suppose that $G_1$ and $G_2$ are graphs such that $V(G_1) \cap V(G_2)$ is both nonempty and a clique in $G_1$ and $G_2$.  Then, the \emph{clique-sum} of $G_1$ and $G_2$, denoted $G_1 \oplus G_2$, is the graph $G$ with $V(G) = V(G_1) \cup V(G_2)$ and $E(G) = E(G_1) \cup E(G_2)$.  Moreover, if $V(G_1) \cap V(G_2)$ is a clique on $k$-vertices in $G_1 \oplus G_2$, then an easy counting argument shows that $P(G_1 \oplus G_2,m) = \frac{P(G_1,m)P(G_2,m)}{m(m-1) \cdots (m-k+1)}$.

The following Proposition will be proven in Section~\ref{bijections}.

\begin{pro} \label{cor: polynomial}
Suppose that $G$ is a graph with $uv \in E(G)$.  Let $e = uv$. If $m \geq 2$ and
$$ P(G- \{e\} , m) < \frac{m}{m-1} P(G,m),$$
then $P_{DP}(G,m) < P(G,m)$.
\end{pro}

This Proposition is the key ingredient in the proof of the following result.

\begin{thm} \label{thm: oddgirth}
Suppose $G_2$ is an arbitrary graph and $G_1 = C_{2k+2}$.  Suppose $G_1$ and $G_2$ share exactly two vertices and one edge, and suppose $G = G_1 \oplus G_2$.  Then, $P_{DP} (G, m) < P(G, m)$ whenever $m \geq \max\{2, \chi(G_2) \}$.
\end{thm}

\begin{proof}
Suppose that $e \in E(G_1) - E(G_2)$ and $m \geq \max\{2, \chi(G_2) \}$.  Then, $G-\{e\}$ is the clique-sum of $G_2$ and a path on $2k+2$ vertices (such that $G_2$ and the path share an edge).  So, we have that
$$P(G - \{e\}, m) =   \frac{P(G_1 - \{e\},m)P(G_2,m)}{m(m-1)} = (m-1)^{2k}P(G_2,m)$$
and
$$P(G, m)= \frac{P(G_1,m)P(G_2,m)}{m(m-1)}=\frac{[(m-1)^{2k+1}+1]P(G_2,m)}{m}.$$
So, we see that:
$$ \frac{m}{m-1} P(G,m) = \left[(m-1)^{2k}+\frac{1}{m-1} \right]P(G_2,m) > (m-1)^{2k}P(G_2,m) = P(G - \{e\}, m). $$
The result follows by Proposition~\ref{cor: polynomial}.
\end{proof}

Notice that Theorem~\ref{thm: oddgirth} implies that if $g \geq 3$ is odd and $G$ consists of an odd cycle on $g$ vertices that shares an edge with an even cycle on $g+1$ vertices, then $G$ has girth $g$ and $P_{DP} (G, m) < P(G, m)$ whenever $m \geq 3$.  We will give an exact formula for the DP color function of graphs that look like $G$ in Section~\ref{exact} (see Theorem~\ref{thm: cyclepluschord}).  We now have the following Corollary.

\begin{customcor} {\bf \ref{cor: girth}}
For any integer $g \geq 3$ there exists a graph $G$ with girth $g$ and an $N \in \N$ such that $P_{DP}(G,m) < P(G,m)$ whenever $m \geq N$.
\end{customcor}

With the above results in mind, it is natural to study the asymptotic behavior of $P(G,m)-P_{DP}(G,m)$ for arbitrary $G$ as $m \rightarrow \infty$.  With this goal in mind, we present a simple lower bound on the DP color function of an arbitrary graph.

\begin{pro} \label{pro: lower}
Suppose $G$ is a graph and $v_1, v_2, \ldots, v_n$ is an ordering of the elements of $V(G)$ such that $v_i$ has precisely $d_i$ neighbors preceding it in the ordering.  If $m > \max_{i \in [n]} d_i$, then
$$P_{DP}(G,m) \geq \prod_{i=1}^n (m-d_i).$$
\end{pro}

\begin{proof}
Assume $\mathcal{H} = (L,H)$ is an arbitrary $m$-fold cover of $G$.  Notice we can greedily construct an $\mathcal{H}$-coloring of $G$ via the following inductive process. Begin by choosing a vertex from $L(v_1)$.  Note that there are $m-d_1 = m$ ways to do this.  Then, for $t \geq 2$, choose a vertex from $L(v_t)$ that is not adjacent (in $H$) to any vertex that has already been chosen.  Since $v_t$ is adjacent (in $G$) to $d_t$ of the vertices: $v_1, v_2, \ldots, v_{t-1}$, there must be at least $m- d_t$ ways to do this.

When our process is complete, our chosen vertices clearly make up an independent set in $H$ of size $n$.  Since at the $i^{th}$ step of the process we have at least $(m-d_i)$ vertices to choose from,
$$P_{DP}(G, \mathcal{H}) \geq \prod_{i=1}^n (m-d_i).$$
The result follows.
\end{proof}

We are now ready to prove Theorem~\ref{thm: asymptotic}

\begin{customthm} {\bf \ref{thm: asymptotic}}
For any graph $G$ with $n$ vertices,
$$P(G,m)-P_{DP}(G,m) = O(m^{n-2}) \; \; \text{as $m \rightarrow \infty$.}$$
\end{customthm}

\begin{proof}
Suppose $P(G,m) = \sum_{i=0}^{n} (-1)^i a_i m^{n-i}$.  Suppose $v_1, v_2, \ldots, v_n$ is an ordering of the elements of $V(G)$ and $v_i$ has $d_i$ neighbors preceding it in the ordering.  Then, if we let
$$\prod_{i=1}^n (m-d_i) = \sum_{i=0}^n b_i m^{n-i},$$
it is clear $b_0 = 1$ and $b_{1} = -\sum_{i=1}^n d_i = -|E(G)|$.  Finally, by Proposition~\ref{pro: lower} and a well-known fact about chromatic polynomials \footnote{$a_0=1$ and $a_1 = |E(G)|$ (see~\cite{W32})}, we see that when $m > \max_{i \in [n]} d_i$,
$$ P(G,m)-P_{DP}(G,m) \leq \sum_{i=0}^{n} (-1)^i a_i m^{n-i} - \sum_{i=0}^n b_i m^{n-i} = \sum_{i=2}^{n} ((-1)^i a_i - b_i) m^{n-i} = O(m^{n-2}).$$
\end{proof}

We end this section by studying the DP color function of chordal graphs.  In contrast to earlier results, the DP color function of a chordal graph equals its chromatic polynomial.

\subsection{Chordal Graphs} \label{chordal}

A \emph{perfect elimination ordering} for a graph $G$ is an ordering of the elements of $V(G)$, $v_1, v_2, \ldots, v_n$, such that for each vertex $v_i$, the neighbors of $v_i$ that occur after $v_i$ in the ordering form  a clique in $G$.  If $v_1, v_2, \ldots, v_n$ is a perfect elimination ordering for the graph $G$, then for each $i \in [n]$, we let $\alpha_i$ denote the number of neighbors of $v_i$ that occur after $v_i$ in the ordering.  For example, $\alpha_n = 0$.

It is well known that a graph $G$ is chordal if and only if there is a perfect elimination ordering for $G$~\cite{FG65}.  Also, if $G$ is chordal and $v_1, v_2, \ldots, v_n$ is a perfect elimination ordering for $G$, $\chi(G) = \max_{i \in [n]} (\alpha_i + 1)$ and there is a simple formula for the chromatic polynomial of $G$~\cite{A03}:
$$ P(G,m) = \prod_{i=1}^n (m - \alpha_i).$$

We are now ready to prove Theorem~\ref{thm: chordal}.

\begin{customthm} {\bf \ref{thm: chordal}}
If $G$ is chordal, then $P_{DP}(G,m) = P(G,m)$ for every $m \in \N$.
\end{customthm}

\begin{proof}
The result is obvious when $m < \chi(G)$.  So, suppose throughout this proof that $m \geq \chi(G)$.  Since $G$ is chordal, we know there is a perfect elimination ordering, $v_1, v_2, \ldots, v_n$, for $G$.

Now, suppose that $\mathcal{H} = (L,H)$ is an arbitrary $m$-fold cover of $G$.  We have that $m \geq \chi(G) = \max_{i \in [n]} (\alpha_i + 1)$.  We can greedily construct an $\mathcal{H}$-coloring of $G$ by the following inductive process.  We color the vertices in the reverse order of the perfect elimination ordering for $G$.  We begin by selecting an element $a_n \in L(v_n)$.  Then, for each $i < n$ we have the following.  Suppose $a_j \in \{a_{i+1}, a_{i+2}, \ldots, a_n \}$.  If $v_iv_j \in E(G)$, then there is at most one vertex in $L(v_i)$ that is adjacent to $a_j$ in $H$, and if $v_iv_j \notin E(G)$, then there are no vertices in $L(v_i)$ adjacent to $a_j$ in $H$.  So, there are at least $m - \alpha_i \geq 1$ vertices in $L(v_i)$ that are not adjacent to any vertices in $\{a_{i+1}, a_{i+2}, \ldots, a_n \}$.  We select such a vertex and call it $a_i$.

It is easy to see that $I = \{a_i : i \in [n] \}$ is an $\mathcal{H}$-coloring of $G$.  Moreover, notice that at each step of the process outlined above we have $m - \alpha_i$ choices for the vertex we choose in $L(v_i)$.  Thus,
$$P(G,m) = \prod_{i=1}^n (m - \alpha_i) \leq P_{DP}(G, \mathcal{H}).$$
Since $\mathcal{H}$ was an arbitrary $m$-fold cover of $G$, it follows that $P(G,m) \leq P_{DP}(G,m)$.
\end{proof}

\section{Natural Bijections and Counting} \label{bijections}

Notice that Theorem~\ref{thm: chordal} applies to trees.  We will now develop a notion that will allow us to show that for any tree, $T$, the $m$-fold covers of $T$ with the fewest DP-colorings have a natural correspondence to the number of proper $m$-colorings of $T$.  This notion will also be key to developing tools that will help us prove some exact formulas for DP color functions of other classes of graphs in Section~\ref{exact}.

Suppose $G$ is a graph and $\mathcal{H} = (L,H)$ is an $m$-fold cover of $G$.  We say there is a \emph{natural bijection between the $\mathcal{H}$-colorings of $G$ and the proper $m$-colorings of $G$} if for each $v \in V(G)$ it is possible to let $L(v) = \{(v,j) : j \in [m] \}$ so that whenever $uv \in E(G)$, $(u,j)$ and $(v,j)$ are adjacent in $H$ for each $j \in [m]$.  Suppose there is a natural bijection between the $\mathcal{H}$-colorings of $G$ and the proper $m$-colorings of $G$.  Note that if $\mathcal{I}$ is the set of $\mathcal{H}$-colorings of $G$ and $\mathcal{C}$ is the set of proper $m$-colorings of $G$, then the function $f: \mathcal{C} \rightarrow \mathcal{I}$ given by
$$f(c) = \{ (v, c(v)) : v \in V(G) \}$$
is a bijection.

\begin{pro} \label{pro: treeDP}
Suppose that $T$ is a tree on $n$ vertices and $\mathcal{H} = (L,H)$ is an $m$-fold cover of $T$ such that $m \geq 2$ and $E_H(L(u),L(v))$ is a perfect matching whenever $uv \in E(T)$.  Then, there is a natural bijection between the $\mathcal{H}$-colorings of $T$ and the proper $m$-colorings of $T$.
\end{pro}

\begin{proof}
Our proof will be by induction on $n$.  Notice that the result is obvious for $n=1,2$.  So, suppose that $n \geq 3$ and the result holds for all natural numbers less than $n$.

Suppose that $v_n$ is a leaf of $T$, and $v_{n-1}$ is the only neighbor of $v_n$ in $T$.  Let $T' = T - \{v_n \}$.  For each $v \in V(T')$, let $L'(v) = L(v)$.  Also, let $H' = H - L(v_n)$.  Then, $\mathcal{H}' = (L',H')$ is an $m$-fold cover of $T'$ such that $E_H(L(u),L(v))$ is a perfect matching whenever $uv \in E(T')$.  The induction hypothesis tells us it is possible for each $v \in V(T')$ to let $L(v) = \{(v,j) : j \in [m] \}$ so that whenever $uv \in E(T')$, $(u,j)$ and $(v,j)$ are adjacent in $H'$ for each $j \in [m]$.  Now, for each $j \in [m]$ let $(v_n,j)$ be the vertex in $L(v_n)$ that is adjacent to $(v_{n-1},j)$ in $H$.  This completes the induction step.
\end{proof}

We now present two tools, Lemmas~\ref{lem: DPedgedelete} and~\ref{lem: DPpathdelete}, that we will use to find exact formulas for DP color functions of graphs that are close to being trees.  In order to develop these tools, we need one basic fact about proper colorings.

\begin{lem} \label{lem: ends}
Suppose that $G$ is a graph with $uv \in E(G)$.  Let $e = uv$.  For each $(i,j) \in [m] \times [m]$, let $C^{(i,j)}_m$ be the set of proper $m$-colorings of $G- \{e\}$ that color $u$ with $i$ and $v$ with $j$.  Then,
\\
(i)  There is an $r \in \N$ such that $|C^{(i,i)}_m|=r$ for each $i \in [m]$.
\\
(ii)  There is a $t \in \N$ such that $|C^{(i,j)}_m|=t$ whenever $i \neq j$ and $i,j \in [m]$.
\\
Consequently, $mr = P(G-\{e\},m) - P(G,m)$ and $m(m-1)t = P(G,m)$.
\end{lem}

Now, we apply Lemma~\ref{lem: ends} along with the notion of natural bijection to prove Lemma~\ref{lem: DPedgedelete} which we will use to determine the DP color function of unicyclic graphs in Section~\ref{exact}.

\begin{lem} \label{lem: DPedgedelete}
Suppose $G$ is a graph and $\mathcal{H} = (L,H)$ is an $m$-fold cover of $G$ with $m \geq 2$.  Suppose $e \in E(G)$ and $e=uv$. Let $H' = H - E_H(L(u),L(v))$ so that $\mathcal{H}' = (L,H')$ is an $m$-fold cover of $G - \{e\}$.  If there is a natural bijection between the $\mathcal{H}'$-colorings of $G-\{e\}$ and the proper $m$-colorings of $G- \{e\}$, then
$$P_{DP}(G, \mathcal{H}) \geq P(G - \{e\},m) - \max \left \{P(G - \{e\},m) - P(G,m) , \frac{P(G,m)}{m-1} \right \}.$$
Moreover, there exists an $m$-fold cover of $G$, $\mathcal{H}^* = (L,H^*)$, such that
$$P_{DP}(G, \mathcal{H}^*) = P(G - \{e\},m) - \max \left \{P(G - \{e\},m) - P(G,m) , \frac{P(G,m)}{m-1} \right \}.$$
\end{lem}

\begin{proof}
We clearly have that there are $P(G - \{e\},m)$ $\mathcal{H}'$-colorings of $G - \{e\}$.  Lemma~\ref{lem: ends} implies that if $i = j$ there are precisely $\frac{1}{m} (P(G - \{e\},m) - P(G,m))$ $\mathcal{H}'$-colorings of $G-\{e\}$ that contain $(u,i)$ and $(v,j)$.  Similarly, if $i \neq j$ there are precisely $\frac{P(G,m)}{m(m-1)}$ $\mathcal{H}'$-colorings of $G-\{e\}$ that contain $(u,i)$ and $(v,j)$.

Since $|E_H(L(u),L(v))| \leq m$, it immediately follows that
$$P_{DP}(G, \mathcal{H}) \geq P(G - \{e\},m) - \max \left \{P(G - \{e\},m) - P(G,m) , \frac{P(G,m)}{m-1} \right \}.$$

Finally, we form $H^*$ as follows.  If $P(G - \{e\},m) - P(G,m) \geq \frac{P(G,m)}{m-1}$, starting from $H'$, draw an edge between $(u,j)$ and $(v,j)$ for each $j \in [m]$.  If $P(G - \{e\},m) - P(G,m) < \frac{P(G,m)}{m-1}$, starting from $H'$, draw an edge between $(u,j)$ and $(v,j+1)$ for each $j \in [m-1]$ and draw an edge between $(u,m)$ and $(v,1)$.  It is clear that in either case the $m$-fold cover $\mathcal{H}^* = (L,H^*)$ has the desired property.
\end{proof}

Lemma~\ref{lem: DPedgedelete} easily implies Proposition~\ref{cor: polynomial} which we used in Section~\ref{asymptotics}.

\begin{custompro} {\bf \ref{cor: polynomial}}
Suppose that $G$ is a graph with $uv \in E(G)$.  Let $e = uv$. If $m \geq 2$ and
$$ P(G- \{e\} , m) < \frac{m}{m-1} P(G,m),$$
then $P_{DP}(G,m) < P(G,m)$.
\end{custompro}

\begin{proof}
We construct an $m$-fold cover of $G$ as follows.  For each $w \in V(G)$ and $j \in [m]$, let $L(w) = \{(w,j) : j \in [m] \}$.  Let $H$ be the graph with vertex set $\bigcup_{w \in V(G)} L(w)$ and edges drawn so that for each $w \in V(G)$, the vertices in $L(w)$ are pairwise adjacent and for each $xy \in E(G)$, $(x,j)$ is adjacent to $(y,j)$ for each $j \in [m]$.  Then, $\mathcal{H} = (L,H)$ is an $m$-fold cover of $G$.  Furthermore, using the notation of Lemma~\ref{lem: DPedgedelete}, there is a natural bijection between the $\mathcal{H}'$-colorings of $G-\{e\}$ and the proper $m$-colorings of $G- \{e\}$.  So, Lemma~\ref{lem: DPedgedelete} implies that there exists an $m$-fold cover of $G$, $\mathcal{H}^* = (L,H^*)$, such that
$$P_{DP}(G, \mathcal{H}^*) = P(G - \{e\},m) - \max \left \{P(G - \{e\},m) - P(G,m) , \frac{P(G,m)}{m-1} \right \}.$$
Note that $P(G- \{e\} , m) < \frac{m}{m-1} P(G,m)$ implies $P(G - \{e\},m) - P(G,m) < \frac{P(G,m)}{m-1}$.  So, we have that
$$P_{DP}(G,m) \leq P_{DP}(G, \mathcal{H}^*) = P(G - \{e\},m) - \frac{P(G,m)}{m-1} < P(G,m).$$
\end{proof}

We now generalize the proof idea of Lemma~\ref{lem: DPedgedelete} in order to obtain another useful tool.

\begin{lem} \label{lem: DPpathdelete}
Suppose $G$ is a graph and $\mathcal{H} = (L,H)$ is an $m$-fold cover of $G$ with $m \geq 3$.  Suppose $\alpha_1, \alpha_2, \alpha_3$ is a path of length two in $G$ and $\alpha_1 \alpha_3 \notin E(G)$.  Let $e_1 = \alpha_1 \alpha_2$ and $e_2 = \alpha_2 \alpha_3$.  Then, let $G_0 = G- \{e_1, e_2 \}$, $G_1 = G- \{e_1 \}$, $G_2 = G- \{e_2 \}$, and $G^*$ be the graph obtained from $G$ by adding an edge between $\alpha_1$ and $\alpha_3$.  Let $H' = H - (E_H(L(\alpha_1),L(\alpha_2)) \cup E_H(L(\alpha_2),L(\alpha_3)))$ so that $\mathcal{H}' = (L,H')$ is an $m$-fold cover of $G_0$.  Suppose that there is a natural bijection between the $\mathcal{H}'$-colorings of $G_0$ and the proper $m$-colorings of $G_0$. Let
\begin{align*}
& A_1 = P(G_0,m)- P(G,m), \\
& A_2 = P(G_0,m) - P(G_2,m) + \frac{1}{m-1} P(G,m), \\
& A_3 = P(G_0,m) - P(G_1,m) + \frac{1}{m-1} P(G,m), \\
& A_4 = \frac{1}{m-1} \left ( P(G_1,m)+P(G_2,m)+P(G^*,m) - P(G,m) \right),  \; \text{and} \\
& A_5 = \frac{1}{m-1} \left ( P(G_1,m)+P(G_2,m) - \frac{1}{m-2} P(G^*,m) \right).
\end{align*}
Then,
$$P_{DP}(G, \mathcal{H}) \geq P(G_0,m) - \max \{A_1, A_2, A_3, A_4, A_5 \}.$$
Moreover, there exists an $m$-fold cover of $G$, $\mathcal{H}^*$, such that
$$P_{DP}(G, \mathcal{H}^*) = P(G_0,m) - \max \{A_1, A_2, A_3, A_4, A_5 \}.$$
\end{lem}

\begin{proof}
We may assume that $|E_H(L(\alpha_1),L(\alpha_2))|=|E_H(L(\alpha_2),L(\alpha_3))|=m$ (since adding edges to $H$ only reduces the number of $\mathcal{H}$-colorings of $G$).  Let $H''$ be the graph with $V(H'') = \bigcup_{i=1}^3 L(\alpha_i)$ and $E(H'')= E_H(L(\alpha_1),L(\alpha_2)) \cup E_H(L(\alpha_2),L(\alpha_3))$.  Clearly $H''$ can be decomposed into $m$ vertex disjoint paths on three vertices of the form: $(\alpha_1,j), (\alpha_2, t), (\alpha_3, r)$ where $j, t, r \in [m]$.  For a given $\mathcal{H}'$-coloring of $G_0$, $I$, the only way that $I$ is not also an $\mathcal{H}$-coloring of $G$ is if $H[I]$ contains at least one edge from one of these aforementioned paths.  For each such path, there are five possibilities for $j$, $t$, and $r$: (1) $j=t=r$, (2) $j=t$ and $j \neq r$, (3) $j \neq t$ and $t=r$, (4) $j \neq t$ and $j = r$, and (5) $j$, $t$, and $r$ are pairwise distinct.  For a given path on three vertices in $H''$ of the form: $(\alpha_1,j), (\alpha_2, t), (\alpha_3, r)$, we will now count the number of $\mathcal{H}'$-colorings of $G_0$ that contain both $(\alpha_1,j)$ and $(\alpha_2,t)$ or contain both $(\alpha_2,t)$ and $(\alpha_3,r)$ in each of the five possible cases.

For case (1) the number of such $\mathcal{H}'$-colorings equals the number of proper $m$-colorings of $G_0$ that color both $\alpha_1$ and $\alpha_2$ with $j$ or color both $\alpha_2$ and $\alpha_3$ with $j$.  Note $P(G,m)$ is the number of proper $m$-colorings of $G_0$ such that both $\alpha_1$ and $\alpha_2$ get different colors and both $\alpha_2$ and $\alpha_3$ get different colors.  So, we get that the number of proper $m$-colorings of $G_0$ that color both $\alpha_1$ and $\alpha_2$ with $j$ or color both $\alpha_2$ and $\alpha_3$ with $j$ is:
$$\frac{1}{m} (P(G_0,m)- P(G,m)) = \frac{1}{m} A_1.$$

For case (2) the number of such $\mathcal{H}'$-colorings equals the number of proper $m$-colorings of $G_0$ that color both $\alpha_1$ and $\alpha_2$ with $j$ or color $\alpha_2$ with $j$ and $\alpha_3$ with $r$.  The number of proper $m$-colorings of $G_1$ that are not proper $m$-colorings of $G$ because $e_1$ is monochromatic is $P(G_1,m)-P(G,m)$.  So, the number of proper $m$-colorings of $G_0$ that color $\alpha_1$ and $\alpha_2$ with $j$ and color $\alpha_3$ with $r$ is $\frac{1}{m(m-1)}(P(G_1,m)-P(G,m))$.  Using Lemma~\ref{lem: ends} and the inclusion-exclusion principle, we get that the number of proper $m$-colorings of $G_0$ that color both $\alpha_1$ and $\alpha_2$ with $j$ or color $\alpha_2$ with $j$ and $\alpha_3$ with $r$ is:
$$ \frac{1}{m} (P(G_0,m) - P(G_2,m)) + \frac{1}{m(m-1)} P(G_1,m) - \frac{1}{m(m-1)}(P(G_1,m)-P(G,m)) = \frac{1}{m} A_2. $$
A similar argument shows that we get $\frac{1}{m} A_3$ such colorings in case (3).

For case (4) the number of such $\mathcal{H}'$-colorings equals the number of proper $m$-colorings of $G_0$ that color $\alpha_1$ with $j$ and $\alpha_2$ with $t$ or color $\alpha_2$ with $t$ and $\alpha_3$ with $j$.  The number of proper $m$-colorings of $G$ that are not proper $m$-colorings of $G^*$ because $\alpha_1 \alpha_3$ is monochromatic is $P(G,m)-P(G^*,m)$.  So, the number of proper $m$-colorings of $G_0$ that color $\alpha_1$ and $\alpha_3$ with $j$ and color $\alpha_2$ with $t$ is $\frac{1}{m(m-1)}(P(G,m)-P(G^*,m))$.  Using Lemma~\ref{lem: ends} and the inclusion-exclusion principle, we get that the number of proper $m$-colorings of $G_0$ that color $\alpha_1$ with $j$ and $\alpha_2$ with $t$ or color $\alpha_2$ with $t$ and $\alpha_3$ with $j$ is:
$$ \frac{1}{m(m-1)} P(G_2,m) + \frac{1}{m(m-1)} P(G_1,m) - \frac{1}{m(m-1)}(P(G,m)-P(G^*,m)) = \frac{1}{m} A_4. $$

For case (5) the number of such $\mathcal{H}'$-colorings equals the number of proper $m$-colorings of $G_0$ that color $\alpha_1$ with $j$ and $\alpha_2$ with $t$ or color $\alpha_2$ with $t$ and $\alpha_3$ with $r$.  The number of proper $m$-colorings of $G_0$ that color $\alpha_1$ with $j$, $\alpha_2$ with $t$, and $\alpha_3$ with $r$ is $\frac{1}{m(m-1)(m-2)}P(G^*,m)$.  Using Lemma~\ref{lem: ends} and the inclusion-exclusion principle, we get that the number of proper $m$-colorings of $G_0$ that color $\alpha_1$ with $j$ and $\alpha_2$ with $t$ or color $\alpha_2$ with $t$ and $\alpha_3$ with $r$ is:
$$ \frac{1}{m(m-1)} P(G_2,m) + \frac{1}{m(m-1)} P(G_1,m) - \frac{1}{m(m-1)(m-2)}P(G^*,m) = \frac{1}{m} A_5. $$
These computations along with the fact that $H''$ can be decomposed into $m$ vertex disjoint paths on three vertices implies that
$$P_{DP}(G, \mathcal{H}) \geq P_{DP}(G_0, \mathcal{H}') - m \cdot \frac{1}{m} \max \{A_i : i \in [5] \}= P(G_0,m) - \max \{A_i : i \in [5] \}$$
as desired.

Finally, the fact that there is an $m$-fold cover of $G$ that allows us to achieve the above lower bound follows from the fact that it is possible to draw the edges in $E_H(L(\alpha_1),L(\alpha_2)) \cup E_H(L(\alpha_2),L(\alpha_3))$ so that all of the $m$ vertex disjoint paths on three vertices in $H''$ have the form described by case (i) where $i \in [5]$.
\end{proof}

\section{Unicyclic Graphs and Cycles with a Chord} \label{exact}

We begin this section by showing how Lemma~\ref{lem: DPedgedelete} can be applied to prove Theorem~\ref{thm: onecycle} (i.e. yield a formula for the DP color function of any unicyclic graph).  A \emph{unicyclic graph} is a connected graph containing exactly one cycle.  It is easy to prove that if $G$ is a unicyclic graph on $n$ vertices that contains a cycle on $t$ vertices, then
$$P(G,m) = (m-1)^{n-t}P(C_t,m) = (m-1)^n + (-1)^t (m-1)^{n-t+1}.$$

\begin{customthm} {\bf \ref{thm: onecycle}}
Suppose $G$ is a unicyclic graph on $n$ vertices.  Then the following statements hold.
\\
(i) For $m \in \N$, if $G$ contains a cycle on $2k+1$ vertices, then $P_{DP}(G,m) = P(G,m)$.
\\
(ii)  For $m \geq 2$, if $G$ contains a cycle on $2k+2$ vertices, then $$P_{DP}(G,m) = (m-1)^n - (m-1)^{n-2k-2}.$$
\end{customthm}

\begin{proof}
Suppose $\mathcal{H} = (L,H)$ is an arbitrary $m$-fold cover of $G$ with $m \geq 2$.  If $uv \in E(G)$, we will assume that $E_H(L(u),L(v))$ is a perfect matching since adding edges to $H$ can only make the number of $\mathcal{H}$-colorings of $G$ smaller.    Suppose $e$ is an edge on the cycle contained in $G$.  Then, $G - \{e\}$ is a tree, and we know that $P(G- \{e\},m)= m(m-1)^{n-1}$.  Proposition~\ref{pro: treeDP} and Lemma~\ref{lem: DPedgedelete}, then imply that
$$P_{DP}(G, \mathcal{H}) \geq m(m-1)^{n-1} - \max \left \{m(m-1)^{n-1} - P(G,m) , \frac{P(G,m)}{m-1} \right \}.$$
Now, if $G$ contains a cycle on $2k+1$ vertices, we know that $P(G,m) = (m-1)^n - (m-1)^{n-2k}$ and
$$P_{DP}(G, \mathcal{H}) \geq m(m-1)^{n-1} - (m(m-1)^{n-1} - [(m-1)^n - (m-1)^{n-2k}])=P(G,m).$$
Since $\mathcal{H}$ was arbitrary, this completes the proof of Statement~(i).  If $G$ contains a cycle on $2k+2$ vertices, we know that $P(G,m) = (m-1)^n + (m-1)^{n-2k-1}$ and
$$P_{DP}(G, \mathcal{H}) \geq  m(m-1)^{n-1} - \frac{P(G,m)}{m-1} = (m-1)^n - (m-1)^{n-2k-2}.$$
This implies that $P_{DP}(G,m) \geq (m-1)^n - (m-1)^{n-2k-2}$.  Finally, Lemma~\ref{lem: DPedgedelete} tells us that there is an $m$-fold cover of $G$, $\mathcal{H}^*$, for which there are precisely $(m-1)^n - (m-1)^{n-2k-2}$ $\mathcal{H}^*$-colorings of $G$.  This completes the proof of Statement~(ii).
\end{proof}

So, if $G$ is a unicyclic graph on $n$ vertices that contains a cycle on $4$ vertices, then
$$P(G,m)-P_{DP}(G,m)= (m-1)^n + (m-1)^{n-3} - [(m-1)^n - (m-1)^{n-4}] = O(m^{n-3}).$$
Asymptotically, we know of no graph with a larger gap between its chromatic polynomial and DP color function than that of a unicyclic graph that contains a cycle on $4$ vertices (see Question~\ref{ques: asymptotic}).

We end this section by showing how Lemma~\ref{lem: DPpathdelete} can be used to find formulas for the DP color function of a cycle plus a chord.  Note that the answer depends on the parity of the lengths of the two maximal cycles properly contained in such a graph.

\begin{thm} \label{thm: cyclepluschord}
The following statements hold.
\\
(i)  Suppose $H = C_{2k+1}$ and $M = C_{2l+1}$ where $l \geq k \geq 1$.  Suppose $H$ and $M$ share exactly two vertices and one edge, and suppose $G = H \oplus M$.  Then, $P_{DP} (G, m) = P(G, m)$ for any $m \in \N$.
\\
(ii)  Suppose $H = C_{2k+2}$ and $M = C_{2l+2}$ where $l \geq k \geq 1$.  Suppose $H$ and $M$ share exactly two vertices and one edge, and suppose $G = H \oplus M$.  Then,
$$P_{DP} (G, m) = \frac{1}{m}[(m-1)^{2k+2l+3}-(m-1)^{2l+1} - (m-1)^{2k+1} - m - 1]$$
whenever $m \geq 3$.
\\
(iii)  Suppose $H = C_{2k+1}$ and $M = C_{2l+2}$ where $l, k \geq 1$.  Suppose $H$ and $M$ share exactly two vertices and one edge, suppose and $G = H \oplus M$.  Then,
$$P_{DP} (G, m) = \frac{1}{m}[(m-1)^{2k+2l+2}-(m-1)^{2k} - (m-1)^{2l+2} + 1]$$
whenever $m \geq 2$.
\end{thm}

\begin{proof}
Note that for Statement~(i) the desired result follows when $k=l=1$ by Theorem~\ref{thm: chordal}.  Since the proof of what remains consists of similar applications of Lemma~\ref{lem: DPpathdelete}, we only present the proof of Statement~(iii).

Since the result is clear when $m=2$, we suppose that $m \geq 3$.  Let $\alpha_1$ and $\alpha_2$ be the vertices in $V(H) \cap V(M)$, and let $\alpha_3$ be the vertex in $V(M)-V(H)$ that is adjacent to $\alpha_2$.  Notice $\alpha_1, \alpha_2, \alpha_3$ is a path of length two in $G$ and $\alpha_1 \alpha_3 \notin E(G)$. Also, suppose that $\mathcal{H} = (L,H)$ is an arbitrary $m$-fold cover of $G$.  To prove the desired result, we first show that $P_{DP}(G, \mathcal{H}) \geq \frac{1}{m}[(m-1)^{2k+2l+2}-(m-1)^{2k} - (m-1)^{2l+2} + 1]$.  If $uv \in E(G)$, we will assume that $E_H(L(u),L(v))$ is a perfect matching since adding edges to $H$ can only make the number of $\mathcal{H}$-colorings of $G$ smaller.

Now, we define $e_1$, $e_2$, $G_0$, $G_1$, $G_2$, $G^*$, and $\mathcal{H}'$ as they are defined in the statement of Lemma~\ref{lem: DPpathdelete}.  Since $G_0$ is a path (and hence a tree), Proposition~\ref{pro: treeDP} implies there is a natural bijection between the $\mathcal{H}'$-colorings of $G_0$ and the proper $m$-colorings of $G_0$.  So, the hypotheses of Lemma~\ref{lem: DPpathdelete} are met.  We use basic facts about chromatic polynomials to compute:
\begin{align*}
&A_1 = m(m-1)^{2k+2l} - \frac{1}{m(m-1)} [(m-1)^{2k+1}-(m-1)][(m-1)^{2l+2}+(m-1)], \\
&A_2 = m(m-1)^{2k+2l} - (m-1)^{2k+2l+1} + (m-1)^{2l+1} + \frac{1}{m-1} P(G,m), \\
&A_3 = m(m-1)^{2k+2l} - (m-1)^{2k+2l+1} + (m-1) + \frac{1}{m-1} P(G,m), \\
&A_4 = 2(m-1)^{2k+2l} - (m-1)^{2l} - 1 - \frac{1}{m}[(m-1)^{2k}-1][(m-1)^{2l}+(m-1)], \; \text{and} \\
&A_5 = 2(m-1)^{2k+2l} - (m-1)^{2l} - 1 - \frac{1}{m}[(m-1)^{2k}-1][(m-1)^{2l}-1]
\end{align*}
It is immediately clear that $A_2 > A_3$ and $A_5 > A_4$.  It is easy to verify that $A_2 > A_1$ and $A_2 > A_5$.  This means $\max \{A_1, A_2, A_3, A_4, A_5 \} = A_2$.  So, Lemma~\ref{lem: DPpathdelete} implies that
$$P_{DP}(G, \mathcal{H}) \geq P(G_0,m) - A_2 = \frac{1}{m}[(m-1)^{2k+2l+2}-(m-1)^{2k} - (m-1)^{2l+2} + 1].$$
Finally, Lemma~\ref{lem: DPpathdelete} also tells us that there is an $m$-fold cover of $G$, $\mathcal{H}^*$, such that $P_{DP}(G,\mathcal{H}^*)$ equals the lower bound above.  The desired result immediately follows.
\end{proof}

\section{DP Color Function of $K_p \vee G$} \label{join}

In this section we study the question whether taking the join of an arbitrary graph with an appropriate clique makes the chromatic polynomial equal to the DP color function. It is easy to see that for any graph $G$, $P(K_p \vee G,m) = \left(\prod_{i=0}^{p-1} (m-i) \right) P(G, m-p)$.  We will now prove Theorem~\ref{thm: joinbound} which we restate.

\begin{customthm} {\bf \ref{thm: joinbound}}
Suppose $G$ is a graph with $\col(G) \geq 3$.  Then for $m \geq \col(G) + 3$,
$$P_{DP}(K_1 \vee G, m) \geq \min \{P(K_1 \vee G,m), mP_{DP}(G, m-1) + 2(m-\col(G)-2)^{|V(G)|-2} \}.$$
\end{customthm}

Throughout this Section, assume $G$ is a graph with $\col(G)=d \geq 3$, and suppose that $v_1, v_2, \ldots, v_n$ is an ordering of the vertices of $G$ such that $v_i$ has at most $d-1$ neighbors preceding it in the ordering.  Also, let $M = K_1 \vee G$, and suppose that $w$ is the vertex corresponding to the copy of $K_1$ used to form $M$.

We will suppose that $\mathcal{H} = (L,H)$ is an $m$-fold cover of $M$ with $m \geq d+3$, and we will assume that $E_H(L(u),L(v))$ is a perfect matching whenever $uv \in E(M)$.  We refer to the edges of $H$ connecting distinct parts of the partition $\{L(v) : v \in V(G) \}$ as \emph{cross-edges}.  We are interested in bounding $P_{DP}(M, \mathcal{H})$ from below.  We may suppose that $L(w) = \{(w,j) : j \in [m] \}$.  For each $j \in [m]$ and $v \in V(G)$, let
$$H^{(j)} = H - N_H[(w,j)] \; \; \text{and} \; \; L^{(j)}(v) = L(v) - \{u : u \in L(v) \cap N_H((w,j)) \}.$$
Then, $\mathcal{H}^{(j)} = (L^{(j)},H^{(j)})$ is an $(m-1)$-fold cover of $G$.  We say that $(w,t) \in L(w)$ is a \emph{level vertex} if $H^{(t)}$ contains precisely $|E(G)|(m-1)$ cross-edges (i.e. $H^{(t)}$ contains the maximum possible number of cross-edges).  We will now prove two lemmas that will immediately imply Theorem~\ref{thm: joinbound}.

\begin{lem} \label{lem: levelvertex1}
If $L(w)$ contains at least $m-1$ level vertices, then there is a natural bijection between the $\mathcal{H}$-colorings of $M$ and the proper $m$-colorings of $M$.  Consequently $P_{DP}(M, \mathcal{H})=P(M,m)$.
\end{lem}

\begin{proof}
We may suppose without loss of generality that $(w,1), (w,2), \ldots, (w,m-1)$ are level vertices.  For each $v \in V(G)$ and $j \in [m]$, call the vertex in $L(v)$ that is adjacent to $(w,j)$ in $H$: $(v,j)$.

Now, we claim that whenever $uv \in E(M)$, $(u,j)$ and $(v,j)$ are adjacent in $H$ for each $j \in [m]$.  This is clear when $u = w$ or $v = w$.  So, suppose $uv \in E(G)$.  For the sake of contradiction, suppose that $(u,t)$ and $(v,t)$ are not adjacent in $H$ for some $t \in [m]$.

Suppose $t \in [m-1]$.  For any $xy \in E(G)$ note that removing one vertex from $L(x)$ and removing one vertex from $L(y)$ deletes one or two edges from $E_H(L(x),L(y))$.  If $(u,t)$ and $(v,t)$ are not adjacent in $H$, then $|E_{H^{(t)}}(L^{(t)}(u),L^{(t)}(v))|=m-2$.  This implies that $H^{(t)}$ contains at most $|E(G)|(m-1)-1$ cross-edges which contradicts the fact that $(w,t)$ is a level vertex.

So, if $(u,t)$ and $(v,t)$ are not adjacent in $H$, we may assume that $t=m$.  By what we just showed, we know that $(u,j)$ and $(v,j)$ are adjacent for each $j \in [m-1]$.  Since $E_H(L(u),L(v))$ is a perfect matching, it must be that $(u,m)$ is adjacent to $(v,m)$ in $H$.  This is a contradiction, and our proof is complete.
\end{proof}

\begin{lem} \label{lem: levelvertex2}
Let $\col(G)=d$.  If $L(w)$ contains $s$ vertices that are not level vertices, then
$$P_{DP}(M, \mathcal{H}) \geq m P_{DP}(G,m-1) + s(m-d-2)^{n-2}.$$
\end{lem}

\begin{proof}
Suppose without loss of generality that $(w,i)$ is not a level vertex for each $i \in [s]$.  Clearly,
$$P_{DP}(M, \mathcal{H}) = \sum_{j=1}^{m} P_{DP}(G, \mathcal{H}^{(j)}) \geq m P_{DP}(G, m-1).$$
To complete the proof we will show that $P_{DP}(G, \mathcal{H}^{(i)}) \geq P_{DP}(G,m-1) + (m-d-2)^{n-2}$ for each $i \in [s]$.

Suppose that $t \in [s]$.  Since $(w,t)$ is not a level vertex, we know that $H^{(t)}$ contains less than $|E(G)|(m-1)$ cross-edges.  So, there exists $v_q, v_r \in V(G)$ such that $q > r$, $v_q v_r \in E(G)$, and $|E_{H^{(t)}}(L^{(t)}(v_q),L^{(t)}(v_r))| = m-2$.  This means there is an $x \in L^{(t)}(v_q)$ and a $y \in L^{(t)}(v_r)$ such that $x$ and $y$ are not saturated by $E_{H^{(t)}}(L^{(t)}(v_q),L^{(t)}(v_r))$.

Let $H^*$ be the graph obtained from $H^{(t)}$ be drawing an edge between $x$ and $y$.  Then, $\mathcal{H}^* = (L^{(t)},H^*)$ is an $(m-1)$-fold cover of $G$, and $P_{DP}(G, \mathcal{H}^*)$ is the number of $\mathcal{H}^{(t)}$-colorings of $G$ that do not include both $x$ and $y$.  Hence, there are at least $P_{DP}(G,m-1)$ $\mathcal{H}^{(t)}$-colorings of $G$ that do not include both $x$ and $y$.

Now, we know that $v_1, v_2, \ldots, v_n$ is an ordering of the vertices of $G$ such that $v_i$ has at most $d-1$ neighbors preceding it in the ordering.  Consider the following ordering of the vertices of $G$:
$$v_r, v_q, v_1, v_2, \ldots, v_{r-1}, v_{r+1}, \ldots, v_{q-1}, v_{q+1}, \ldots, v_n.$$
In this ordering each vertex has at most $d+1$ neighbors preceding it in the ordering.  Thus, there are at least $((m-1)-(d+1))^{n-2} = (m-d-2)^{n-2}$ $\mathcal{H}^{(t)}$-colorings of $G$ that include both $x$ and $y$.  This immediately implies $P_{DP}(G, \mathcal{H}^{(t)}) \geq P_{DP}(G,m-1) + (m-d-2)^{n-2}$.
\end{proof}

Having proven Theorem~\ref{thm: joinbound}, we have the following Corollary that will allow us to easily prove Theorem~\ref{thm: joinasy}.

\begin{cor} \label{cor: joinbound}
Suppose $G$ is a graph with $n$ vertices and $\col(G) \geq 3$.  Then, for any $p \in \N$ and $m \geq \col(G) + 2 + p$,
$$P_{DP}(K_p \vee G, m) \geq \min \left \{P(K_p \vee G,m), \left(\prod_{j=0}^{p-1} (m-j) \right) P_{DP}(G, m-p) + f(m) \right \}$$
where $f(m)$ is a polynomial in $m$ of degree $n-3+p$ with a leading coefficient of $2p$.
\end{cor}

\begin{proof}
The proof is by induction on $p$.  Notice that the base case is Theorem~\ref{thm: joinbound}.  So, suppose that $p \geq 2$, and the result holds for all natural numbers less than $p$.

Suppose that $m$ satisfies $m \geq \col(G) + 2 + p$.  Since $K_p \vee G = K_1 \vee (K_{p-1} \vee G)$ and $\col(K_{p-1} \vee G) \leq \col(G) + p-1$, Theorem~\ref{thm: joinbound} tells us
$$P_{DP}(K_p \vee G, m) \geq \min \{P(K_p \vee G,m), mP_{DP}(K_{p-1} \vee G, m-1) + 2(m-\col(K_{p-1} \vee G)-2)^{n-3+p} \}.$$
Since $m-1 \geq \col(G) + 2 + p - 1$, the inductive hypothesis tells us
$$P_{DP}(K_{p-1} \vee G, m-1) \geq \min \left \{P(K_{p-1} \vee G, m-1), \left(\prod_{j=1}^{p-1} (m-j) \right) P_{DP}(G, m-p) + f(m) \right \}$$
where $f(m)$ is a polynomial in $m$ of degree $n-4+p$ with a leading coefficient of $2(p-1)$.

In the case that $P(K_{p-1} \vee G, m-1) \leq \left(\prod_{j=1}^{p-1} (m-j) \right) P_{DP}(G, m-p) + f(m)$, we have that $P_{DP}(K_{p-1} \vee G, m-1) =P(K_{p-1} \vee G, m-1)$.  This means that $mP_{DP}(K_{p-1} \vee G, m-1) = m P(K_{p-1} \vee G, m-1)  =P(K_p \vee G,m)$.  So,
$$P_{DP}(K_p \vee G, m) \geq \min \{P(K_p \vee G,m), mP_{DP}(K_{p-1} \vee G, m-1) + 2(m-\col(K_{p-1} \vee G)-2)^{n-3+p} \}$$
implies $P_{DP}(K_p \vee G, m) = P(K_p \vee G,m)$ and the desired result follows.  So, we may assume that $P(K_{p-1} \vee G, m-1) > \left(\prod_{j=1}^{p-1} (m-j) \right) P_{DP}(G, m-p) + f(m)$.  We calculate that:
\begin{align*}
& mP_{DP}(K_{p-1} \vee G, m-1) + 2(m-\col(K_{p-1} \vee G)-2)^{n-3+p} \\
& \geq m \left( \left(\prod_{j=1}^{p-1} (m-j) \right) P_{DP}(G, m-p) + f(m)  \right) + 2(m-\col(K_{p-1} \vee G)-2)^{n-3+p} \\
&= \left(\prod_{j=0}^{p-1} (m-j) \right) P_{DP}(G, m-p) + mf(m) + 2(m-\col(K_{p-1} \vee G)-2)^{n-3+p}.
\end{align*}
If we let $g(m) = mf(m) + 2(m-\col(K_{p-1} \vee G)-2)^{n-3+p}$, then $g(m)$ is a polynomial in $m$ of degree $n-3+p$ with a leading coefficient of $2(p-1)+2=2p$.  Furthermore, we have
$$P_{DP}(K_p \vee G, m) \geq \min \left \{P(K_p \vee G,m), \left(\prod_{j=0}^{p-1} (m-j) \right) P_{DP}(G, m-p) + g(m) \right \}$$
which completes the induction step.
\end{proof}

We now prove Theorem~\ref{thm: joinasy}.

\begin{customthm} {\bf \ref{thm: joinasy}}
Suppose $G$ is a graph on $n$ vertices such that
$$P(G,m)-P_{DP}(G,m) = O(m^{n-3}) \; \; \text{as $m \rightarrow \infty$.}$$
Then, there exist $p,N \in \N$ such that $P_{DP}(G \vee K_p, m) = P(G \vee K_p , m)$ whenever $m \geq N$.
\end{customthm}

\begin{proof}
We may assume that $\col(G) \geq 3$ since the result is obvious when $\col(G) \leq 2$.  We know there are constants $C, N_1 \in \N$ such that
$$P(G,m) - P_{DP}(G,m) \leq Cm^{n-3}$$
whenever $m \geq N_1$.  Fix $p$ as a natural number such that $p > C/2$.  Then, for $m \geq p + N_1$, we have that
\begin{align*}
\left(\prod_{j=0}^{p-1} (m-j) \right) P_{DP}(G, m-p) & \geq \left(\prod_{j=0}^{p-1} (m-j) \right) \left(P(G,m-p) - C(m-p)^{n-3} \right) \\
&= P(K_p \vee G,m) - C(m-p)^{n-3} \prod_{j=0}^{p-1} (m-j).
\end{align*}
Corollary~\ref{cor: joinbound} implies that for $m \geq \col(G) + 2 + p$,
$$P_{DP}(K_p \vee G, m) \geq \min \left \{P(K_p \vee G,m), \left(\prod_{j=0}^{p-1} (m-j) \right) P_{DP}(G, m-p) + f(m) \right \}$$
where $f(m)$ is a polynomial in $m$ of degree $n-3+p$ with a leading coefficient of $2p$.  Since $p > C/2$ we know that
$$f(m) - C(m-p)^{n-3} \prod_{j=0}^{p-1} (m-j)$$
is a polynomial of degree $n-3+p$ with a positive leading coefficient.  Thus, there must be an $N \in \N$ such that
\begin{align*}
P(K_p \vee G,m) &\leq P(K_p \vee G,m) + f(m) - C(m-p)^{n-3} \prod_{j=0}^{p-1} (m-j) \\
& \leq \left(\prod_{j=0}^{p-1} (m-j) \right) P_{DP}(G, m-p) + f(m)
\end{align*}
for each $m \geq N$.  The result follows by Corollary~\ref{cor: joinbound}.
\end{proof}

{\bf Acknowledgment.} The authors would like to thank Alexandr Kostochka for his helpful comments on this paper.  The authors would also like to thank Jade Hewitt, David Spivey, and Seth Thomason for discussions on Question~\ref{ques: mono}.

\end{document}